\documentclass[12pt]{article}
\usepackage{float}
\usepackage{makeidx}
\usepackage{latexsym}

\usepackage[margin=2.5cm]{geometry}
\usepackage[dvips]{graphicx}
\usepackage{amssymb}
\usepackage{amsmath}
\usepackage{amsthm}
\usepackage{lineno}
\usepackage{tikz}
\usetikzlibrary{arrows}
\theoremstyle{plain}
\newtheorem{theorem}{Theorem}[section]

\newtheorem{claim}[theorem]{Claim}

\newtheorem{lemma}[theorem]{Lemma}
\newtheorem{corollary}[theorem]{Corollary}
\newtheorem{conjecture}[theorem]{Conjecture}
\newtheorem{proposition}[theorem]{Proposition}

\newtheorem{problem}[theorem]{Problem}

\theoremstyle{definition} 
\newtheorem{remark}[theorem]{Remark}

\begin{document}
\date{\today}
\title{Avoiding and extending partial edge colorings of hypercubes\footnote{
This paper is partially based on the Bachelor thesis of Johansson written
under the supervision of Casselgren.}}

\author{
{\sl Carl Johan Casselgren}\footnote{Department of Mathematics, Link\"oping University, SE-581 83 Link\"oping, Sweden. 
%\newline
{\it E-mail address:} carl.johan.casselgren@liu.se.
\, Casselgren was supported by a grant from the
Swedish Research Council (2017-05077)}
\and {\sl Per Johansson}\footnote{Department of Mathematics, Link\"oping University, SE-581 83 Link\"oping, Sweden. %\newline
{\it E-mail address:} perjo018@student.liu.se  }
\and {\sl Klas Markstr\"om}\footnote{Department of Mathematics, 
Ume\aa\enskip University, 
SE-901 87 Ume\aa, Sweden.
{\it E-mail address:} klas.markstrom@umu.se
}
}
\maketitle

\bigskip
\noindent
{\bf Abstract.}
	We consider the problem of extending and avoiding partial edge colorings
	of hypercubes; that is, given a partial edge coloring $\varphi$ of 
	the $d$-dimensional hypercube $Q_d$,
	we are interested in whether there is a proper $d$-edge coloring of $Q_d$
	that {\em agrees} with the coloring $\varphi$ on every edge
	that is colored under $\varphi$; or, similarly, if there is a proper
	$d$-edge coloring that {\em disagrees} with $\varphi$ on every edge
	that is colored under $\varphi$.
	In particular, we prove that
	for any $d\geq 1$, 
	if $\varphi$ is a partial $d$-edge coloring of $Q_d$, then
	$\varphi$ is avoidable if
	every color appears on at most $d/8$ edges and 
	the coloring satisfies a relatively
	mild structural condition, or $\varphi$ is
	proper and every color appears on at most $d-2$ edges. 
	We also show that the same conclusion holds if $d$ is divisible by
	$3$ and every color class of
	$\varphi$ is an induced matching.	
	Moreover, for all $1 \leq k \leq d$, we characterize
	for which configurations consisting of a partial coloring 
	$\varphi$ of $d-k$ edges 
	and a partial coloring $\psi$ of $k$ edges, 
	there is an extension
	of $\varphi$ that avoids $\psi$.

\bigskip


\section{Introduction}

	An {\em edge precoloring} (or {\em partial edge coloring}) 
	of a graph $G$ is a proper edge coloring of some 
	subset $E' \subseteq E(G)$; {\em a $t$-edge precoloring}
	is such a coloring with $t$ colors.	
		A $t$-edge precoloring $\varphi$ is
	{\em extendable} if there is a proper $t$-edge coloring $f$
	such that $f(e) = \varphi(e)$ for any edge $e$ that is colored
	under $\varphi$; $f$ is called an {\em extension}
	of $\varphi$. %{\red Unless otherwise stated, all proper $t$-edge
	%colorings in this paper use colors $1,\dots, t$.}
	
	Related to the notion of extending a precoloring is the idea of
	{\em avoiding} a precoloring: if $\varphi$ is an edge precoloring
	of a graph $G$, then a proper edge coloring $f$ of $G$ {\em avoids}
	$\varphi$ if $f(e) \neq \varphi(e)$ for every $e \in E(G)$.
	More generally, if $L$ is a list assignment for the edges of a graph $G$,
	then a proper  edge coloring $\varphi$ of $G$ {\em avoids}
	the list assignment $L$ if $\varphi(e) \notin L(e)$ for 
	every edge $e$ of $G$.
	
	In general, the problem of extending a given edge precoloring
	is an $\mathcal{NP}$-complete problem, 
	already for $3$-regular bipartite graphs \cite{Colbourn, Fiala}.
	One of the earlier references explicitly discussing the problem of 
	extending a partial  edge coloring is \cite{MarcotteSeymour}; there a 
	necessary condition for the existence of an extension is 
	given and the authors  find a class  of graphs where this 
	condition is also sufficient.
	More recently, questions on extending and avoiding a
	precolored matching have been studied in
	\cite{EGHKPS,GiraoKang}.
	In particular, in \cite{EGHKPS}
	it is proved that if $G$ is subcubic or bipartite
	and $\varphi$ is an edge precoloring of a matching $M$ in $G$
	using $\Delta(G)+1$ colors, then $\varphi$ 
	can be extended to a proper $(\Delta(G)+1)$-edge coloring of $G$,
	where $\Delta(G)$ as usual denotes the maximum degree of $G$;
	a similar result on avoiding a precolored matching of a general graph
	is obtained as well.
	%as long as {\blue $\Delta(G)+1$} colours are allowed; 
	%by the {\em distance} between two
	%edges $e$ and $e'$ we mean the number of edges in a 
	%shortest path between an endpoint of $e$ and an endpoint of $e'$; 
	%a {\em distance-$t$ matching} is a matching where any
	%two edges are at distance at least $t$ from each other. A
	%distance-$2$ matching is also called an {\em induced matching}.
	Moreover, in \cite{GiraoKang} it is proved that
	if $\varphi$ is an $(\Delta(G)+1)$-edge precoloring of a 
	distance-$9$ matching in
	any graph $G$, then $\varphi$ can be extended
	to a proper $(\Delta(G)$+1)-edge coloring of $G$; here, by a 
	{\em distance-9-matching} we mean a matching $M$ where the distance between
	any two edges in $M$ is at least $9$; the {\em distance} between two
	edges $e$ and $e'$ is the number of edges contained in a shortest path 
	between an endpoint of $e$, and an endpoint of $e'$. 
	A distance-$2$ matching is usually called an
	{\em induced matching}.

	Questions on extending and avoiding partial edge colorings
	have specifically been studied to a large extent for balanced complete
	bipartite graphs, usually formulated in terms of completing partial
	Latin squares and avoiding arrays, respectively.
	In this form, the problem goes back to the famous Evans conjecture
	\cite{Evans} which states that for every positive integer $n$, 
	if $n-1$ edges in the complete bipartite graph 
	$K_{n,n}$ have been (properly) colored, 
	then this partial coloring can be extended to a
	proper $n$-edge coloring of $K_{n,n}$.
	This conjecture was solved for large $n$ by H\"aggkvist \cite{Haggkvist78}
	and later for all $n$ by Smetaniuk \cite{Smetaniuk}, 
	and independently by Andersen and Hilton \cite{AndersenHilton}.

	The problem of avoiding partial edge colorings (and list assignments) of
	complete bipartite graphs was introduced by
	Häggkvist \cite{Haggkvist}
	and has been further studied in e.g 
	\cite{AndrenEven, AndrenOdd, Casselgren}.
	In particular, by results of \cite{Cavenagh, ChetwyndRhodes, Ohman},
	any partial proper $n$-edge coloring of $K_{n,n}$ is avoidable,
	given that $n \geq 4$.
	Moreover, a conjecture first stated by
	Markstr\"om 
	suggests that if $\varphi$ is a partial $n$-edge coloring
	of $K_{n,n}$, where any color appears on at most $n-2$ edges, then $\varphi$
	is avoidable (see e.g. \cite{Casselgren}).
	In \cite{Casselgren},
	several partial results towards this conjecture are obtained;
	in particular, it is proved that the conjecture holds if each color
	appears on at most $n/5$ edges, or if the graph is colored
	by altogether at most $n/2$ colors.
	
	Combining the notion of extending a precoloring and avoiding
	a list assignment, Andren et al. \cite{AndrenCasselgrenMarkstrom}
	proved that 
	every ``sparse'' partial edge
	coloring of $K_{n,n}$ can be extended %in such a way that
	to a proper $n$-edge coloring avoiding a given 
	list assignment $L$ satisfying certain ``sparsity'' conditions, 
	provided that no edge $e$ is precolored by a color
	that appears in $L(e)$; we refer to \cite{AndrenCasselgrenMarkstrom} 
	for the exact definition of ``sparse'' in this context.
	An analogous result for
	complete graphs was recently obtained in \cite{CasselgrenPham}.

	The study of problems on extending and avoiding partial edge 
	colorings of hypercubes
	was recently initiated in the papers 
	\cite{CasselgrenMarkstromPham, CasselgrenMarkstromPham2}. In
	\cite{CasselgrenMarkstromPham} Casselgren et al 
	obtained several analogues for hypercubes of classic results
	on completing partial Latin squares, such as the famous Evans conjecture.
	Moreover, questions on extending a ``sparse'' precoloring of a hypercube
	subject to the condition that the extension should avoid a given
	``sparse'' list assignment were investigated in 
	\cite{CasselgrenMarkstromPham2}.

	\bigskip

In this paper we continue the work on extending and avoiding partial edge colorings
of hypercubes, with a particular focus on the latter variant. 
We obtain a number of results towards an analogue for hypercubes of 
Markstr\"om's aforementioned
conjecture for complete bipartite graphs (see Conjecture \ref{conj:general}), 
and also prove several related results;
in particular, we prove the following.

\begin{itemize}

	\item For any $d\geq 1$, if $\varphi$ is a partial $d$-edge coloring of
	$Q_d$ where every color appears on at most $d/8$ edges, and $\varphi$
	satisfies
	a structural condition (described in Theorem \ref{th:gen} below),	
	then $\varphi$ is avoidable;
	
	\item for any $d \geq 1$, if $\varphi$ is a partial proper $d$-edge coloring
	of $Q_d$ where every color appears on at most $d-2$ edges, then $\varphi$
	is avoidable;
	
	\item if $d=3k$, where $k \geq 1$ is a positive integer, and every
	color class of the partial $d$-edge coloring $\varphi$ of $Q_d$ 
	is an induced matching, then $\varphi$ is avoidable;
	we conjecture that this holds for any $d \geq 1$;
	
	\item for any $d\geq 1$ and any $1 \leq k \leq d$, we characterize
	for which configurations consisting of a partial coloring 
	$\varphi$ of $d-k$ edges and a partial coloring $\psi$ of $k$ edges, 
	there is an extension
	of $\varphi$ that avoids $\psi$.

\end{itemize}

\section{Preliminaries}

	In this paper, all (partial) $d$-edge colorings use colors $1,\dots, d$
	unless otherwise stated.
	If $\varphi$ is an edge precoloring of $G$,
	and an edge $e$ is colored under $\varphi$,
	then we say that $e$ is {\em $\varphi$-colored}.

	If $\varphi$ is a (partial) proper $t$-edge coloring of $G$
	and $1 \leq a,b \leq t$, then a path or cycle in
	$G$ is called {\em $(a,b)$-colored under $\varphi$} if
	its edges are colored by colors $a$ and $b$ alternately.
	We also say that such a path or cycle is \emph{bicolored under $\varphi$}.
	By switching colors
	$a$ and $b$ on a maximal $(a,b)$-colored path or an $(a,b)$-colored cycle,
	we obtain another proper $t$-edge coloring of $G$;
	this operation is called an {\em interchange}.
	We denote by $\varphi^{-1}(i)$ the set of edges colored $i$ under $\varphi$.
	
	In the above definitions, we often leave out the explicit 
	reference to a
	coloring $\varphi$, if the coloring
	is clear from the context.

	\bigskip
	
	Havel and Moravek \cite{HavelMoravek} (see also \cite{Harary})
	proved a criterion
	for a graph $G$ to be a subgraph of a hypercube:

\begin{proposition}
\label{prop:HavelMoravek}
	A graph $G$ is a subgraph of $Q_d$ if and only if 
	there is a proper $d$-edge coloring of 
	$G$ with integers $\{1,\dots,d\}$
	such that
	\begin{itemize}
	%\item[(i)] edges incident with a common vertex have different labels;
	%
	\item[(i)] in every path of $G$ there is some color that appears an odd
	number of times;
	
	\item[(ii)] in every cycle of $G$ no color appears an odd number of times.
	\end{itemize}
\end{proposition}

	A {\em dimensional matching} $M$ of $Q_d$ is a perfect matching of $Q_d$
	such that $Q_d- M$ is isomorphic to two copies of $Q_{d-1}$; evidently
	there are precisely $d$ dimensional matchings in $Q_d$.
	We state this as a lemma.
	
	\begin{lemma}
	\label{lem:matchings}
		Let $d \geq 2$ be an integer.
		Then there are $d$ different dimensional matchings in $Q_d$;
		indeed $Q_d$ decomposes into $d$ such perfect matchings.
	\end{lemma}
	The proper $d$-edge coloring of $Q_d$
	obtained by coloring the $i$th dimensional matching of $Q_d$ by color $i$,
	$i=1,\dots,d$, we shall refer to as the {\em standard edge coloring} of $Q_d$.

	As pointed out in \cite{CasselgrenMarkstromPham}, the colors in
	the proper edge coloring in Proposition \ref{prop:HavelMoravek}
	correspond to dimensional matchings in $Q_d$ (see also \cite{Harary}).
	In particular, Proposition \ref{prop:HavelMoravek} holds if we take the
	dimensional matchings as the colors. Furthermore we have the following.
	
	\begin{lemma}
	\label{lem:DimMathInduce}
		The subgraph induced by $r$ dimensional matchings
		in $Q_d$ is isomorphic to a disjoint union of
		$r$-dimensional hypercubes.
	\end{lemma}
	
	This simple observation shall be used quite frequently below.
	In particular, for future reference, we state the following
	consequence of Lemma \ref{lem:DimMathInduce}.
	
	\begin{lemma}
	\label{lem:4cycle}
		In the standard $d$-edge coloring, every edge of $Q_d$ 
		is in exactly $d-1$ $2$-colored $4$-cycles.
	\end{lemma}

	We shall also need some standard definitions
	on list edge coloring.
	Given a graph $G$, assign to each edge $e$ of $G$ a set
	$\mathcal{L}(e)$ of colors.
	%Such an assignment $\mathcal{L}$ is called
	%a \emph{list assignment} for $G$ and
	%the sets $\mathcal{L}(e)$ are referred
	%to as \emph{lists} or \emph{color lists}.
	If all lists have equal size $k$, then $\mathcal{L}$
	is called a \emph{$k$-list assignment}.
	Usually, we seek a proper
	edge coloring $\varphi$ of $G$,
	such that $\varphi(e) \in \mathcal{L}(e)$ for all
	$e \in E(G)$. If such a coloring $\varphi$ exists then
	$G$ is \emph{$\mathcal{L}$-colorable} and $\varphi$
	is called an \emph{$\mathcal{L}$-coloring}. 
	Denote by $\chi'_L(G)$ the minimum integer $t$
	such that $G$ is $\mathcal{L}$-colorable
	whenever $\mathcal{L}$ is a $t$-list assignment.
	A fundamental result in list edge coloring theory
	is the following theorem by Galvin
	\cite{Galvin}. As usual, $\chi'(G)$ denotes the chromatic
	index of a multigraph $G$.

\begin{theorem}
\label{th:Galvin}
	For any bipartite multigraph $G$, $\chi_L'(G) = \chi'(G)$.
\end{theorem}

\section{Avoiding general partial edge colorings}

Most of the results of this paper are partial results towards the following
general conjecture for hypercubes. This is a variant of a conjecture for $K_{n,n}$
first suggested by Markstr\"om based on unavoidable $n$-edge colorings
of $K_{n,n}$ (see e.g. \cite{Casselgren, MarkstromOhman}).

\begin{conjecture}
\label{conj:general}
	For any $d\geq 1$, if $\varphi$ is a partial $d$-edge coloring
	of $Q_d$ where every color appears on at most $d-2$ edges, then
	$\varphi$ is avoidable.
\end{conjecture}

Conjecture \ref{conj:general} is best possible: consider the partial coloring of $Q_d$ 
obtained by coloring $d-1$ edges incident with a vertex $u$ by the color $1$,
and coloring $d-1$ edges incident with another vertex $v$ by the color $2$.
This partial coloring is unavoidable if $uv \in E(Q_d)$ and it is uncolored.

Note further that such a statement as in Conjecture \ref{conj:general} 
does not hold for
general $d$-regular (bipartite) graphs.
Indeed, we have the following:

\begin{proposition}
\label{prop:counter}
	For any $d \geq 1$, there is a $d$-regular bipartite graph $G$ and a 
	partial proper $d$-edge coloring with exactly
	$d$ colored edges that is not avoidable.
\end{proposition}
\begin{proof}
	The case when $d=1$ is trivial, so assume that $d \geq 2$.
	Let $G_1,\dots, G_d$ be $d$ copies of the graph $K_{d,d}-e$, that is, the complete
	bipartite graph $K_{d,d}$ with an arbitrary edge $e$ removed.
	Denote by $a_ib_i$ the edge that was removed from $K_{d,d}$ to 
	form the graph $G_i$.
	From $G_1,\dots, G_d$, we construct the $d$-regular bipartite 
	graph $G$ by adding the edges
	$a_1b_2, a_2b_3,\dots, a_{d-1} b_d, a_d b_1$. 
	
	We define a partial $d$-edge coloring $\varphi$ of $G$ by coloring $a_ib_{i+1}$
	by the color $i$, $i=1,\dots,d$ (where indices are taken modulo $d$).
	Now, it is straightforward that any proper $d$-edge coloring of $G$ uses
	the same color on all the edges in the set 
	$\{a_1b_2, a_2b_3, \dots, a_{d-1} b_d, a_d b_1\}$; therefore, $\varphi$
	is not avoidable.
\end{proof}

On the other hand, a partial coloring of at most $d-1$ edges of a
$d$-edge-colorable graph is always avoidable:

\begin{proposition}
\label{prop:multi}
Let $k\in \{1,\dots,d\}$
and let $G$ be a $d$-edge-colorable graph.
If $G$ is colored with at most $k$ colors,
and every color appears on at most $d-k$ edges, then
there is a proper $d$-edge coloring of $G$
that avoids the preassigned colors.
\end{proposition}
%4.1 in \cite{Casselgren}

This is a reformulation for general graphs of a theorem in \cite{Casselgren}
for complete bipartite graphs;
the proof is identical to the argument given there; thus, we omit it.

Note further that Proposition \ref{prop:multi} does not
set any restrictions on where colors may appear,
so several colors may be assigned
to the same edge. Thus, it has a natural interpretation as a statement on list
edge coloring.

By the example preceding Proposition \ref{prop:multi}, it is in general sharp; however,
by requiring that the colored edges satisfy some structural condition,
we can prove that other configurations are avoidable as well.

%As in \cite{Casselgren}, we can prove similar results by instead
%assuming that all colored edges are incident with at most $k$ vertices.

\begin{proposition}
	Let $G$ be a $d$-edge colorable graph.
	If $\varphi$ is a partial $d$-edge coloring of $G$,
	and there is a set $K$ of $k$ vertices such that every precolored
	edge is incident to some vertex from $K$,
	and every color occurs on at most $d-k$ edges,
	then $\varphi$ is avoidable.
\end{proposition}

	The proof of this proposition is similar to the proof of the previous one.
	The only essential difference is that instead of using the fact
	that the precoloring uses at most $k$ colors,	
	one employs the property that
	every matching in a decomposition obtained
	from a proper $k$-edge coloring of $G$
	contains edges with at most $k$ distinct colors from $\varphi$; 
	we omit the details.

\bigskip

Next, we prove the following weaker version of Conjecture \ref{conj:general}.
Following \cite{CasselgrenMarkstromPham2}, we say that 
two edges in a hypercube are {\em parallel} if they are non-adjacent and
contained in a common $4$-cycle.

We shall use the following simple
lemma.

\begin{lemma}
\label{lem:onecolor}
	If $\varphi$ is a partial $d$-edge coloring of $Q_d$, $d \geq 3$, 
	where every color
	appears on at most one edge, then $\varphi$ is avoidable.
\end{lemma}

\begin{proof}
	Let $f$ be the proper $d$-edge coloring of $Q_d$ obtained by assigning
	color $i$ to the $i$th
	dimensional matching of $Q_d$, that is,
	$f$ is the standard edge coloring of $Q_d$.
	
	Consider the
	bipartite graph $B(f)$, with vertices for the colors $\{1,\dots,d\}$
	and for the color classes $f^{-1}(i)$
	%(perfect matchings)
	of $f$, and where
	there is an edge between $f^{-1}(i)$ and $j$ if there is
	no edge colored $i$ under $f$ that is colored $j$ under $\varphi$.
	If 
	there is no set violating Hall's condition 
	for a matching in a bipartite graph, then $B(f)$ has a perfect matching,
	and by assigning colors to the color classes of $f$ according to this
	perfect matching, we obtain a proper $d$-edge coloring of $Q_d$ that 
	avoids $\varphi$.

	Now, if there is such a set violating Hall's condition, then one of the 
	color classes of $f$ contains all $\varphi$-colored edges.
	Without loss of generality, assume that $M_1$ is such a color class
	and consider the subgraph $H=Q_d[M_1 \cup M_2]$, where $M_2$ is another
	arbitrarily chosen color class of $f$.
	By Lemma \ref{lem:DimMathInduce}, $H$ consists of a collection 
	of bicolored $4$-cycles. By interchanging colors on such a bicolored
	cycle that contains at least one $\varphi$-colored edge, we obtain 
	a proper edge coloring $f'$ of $Q_d$ such that the bipartite graph
	$B(f')$, defined as above, contains
	a perfect matching. Thus there is a proper $d$-edge coloring
	that avoids $\varphi$.
\end{proof}

\begin{theorem}
\label{th:gen}
	Let $d\geq 1$, and let $\varphi$ be a partial $d$-edge coloring
	of $Q_d$. Assume $a(d)$ and $b(d)$ are functions satisfying
	that $\frac{109}{1776} d^2 - 2 b(d)\left(a(d) - \frac{7d}{8}\right)\geq 0$
	and $a(d) \geq b(d)$. 
	\begin{itemize}
	
	\item[(i)] If every color appears on at most $d/8$ edges, for every
	edge in $Q_d$ there is at most $b(d)$ other parallel $\varphi$-colored
	edges, and every dimensional matching in $Q_d$ contains at most
	$a(d)$ $\varphi$-precolored edges, then $\varphi$ is avoidable.
	
	\item[(ii)] For every constant $C_1 \geq 1$, there is a positive constant
	$C_2= 2 C_1(C_1+2)34^2$, 
	such that if every dimensional matching contains at most
	$C_1 d$ $\varphi$-colored edges and every color 
	appears on at most $\frac{d}{C_2}$
	edges under $\varphi$, then $\varphi$ is avoidable.
	
	\end{itemize}
\end{theorem}

Before proving Theorem \ref{th:gen}, allow us to comment on the possible
values of $a(b)$ and $b(d)$ for which the inequality in the theorem holds.
If we choose $b(d)$ to be as large as possible, that is, $a(d) =b(d)$,
then it suffices to require that $a(d) = b(d) \leq
\left(\left(\frac{109}{3552}+\frac{49}{96}\right)^{1/2}+ \frac{7}{16}\right)d
\approx 1.17 d$ for part (i) of
the theorem to hold.
On the other hand, if $b(d)$ is a ``sufficiently small'' linear function of $d$, then we can pick
$a(d)$ to be an arbitrarily large linear function of $d$.

\begin{proof}[Proof of Theorem \ref{th:gen}]
We first prove part (i) of the theorem.
Let $f$ be the standard edge coloring of $Q_d$.
As in the proof of the preceding lemma, 
our goal is to transform $f$ into a coloring $f'$ where
every color class contains edges of at most $\frac{7}{8} d$
distinct colors under $\varphi$, using interchanges on $2$-colored $4$-cycles;
this ensures that in the
bipartite graph $B$, with vertices for the colors $\{1,\dots,d\}$
and for the color classes $f'^{-1}(i)$
%(perfect matchings)
of $f'$, and where
there is an edge between $f'^{-1}(i)$ and $j$ if there is
no edge colored $i$ under $f'$ that is colored $j$ under $\varphi$,
there is no set violating Hall's condition for a matching in a bipartite graph. 
Hence, $B$ has a perfect matching, and by coloring the color classes
of $f'$ according to the perfect matching, 
we obtain a proper $d$-edge coloring of $Q_d$
that avoids $\varphi$.

We shall use the following method for obtaining such a proper
coloring $f'$ from $f$. Suppose that there is some color class
of $f$ that contains at least $\frac{7}{8}d+1$ edges that are colored
under $\varphi$; let $M_1=f^{-1}(1)$ be such a color class.
We call such
a color class {\em heavy}; a color class 
that contains at most $\frac{7}{8}d-2$ edges that are colored
under $\varphi$ is called a {\em light} color class.

Since there are at most $\frac{1}{8}d^2$ edges in $Q_d$ that are colored
under $\varphi$, there must be some light color class of $f$;
without loss of generality assume that $M_2=  f^{-1}(2)$ is such a color class.
By Lemma \ref{lem:DimMathInduce}, 
the subgraph $Q_d[M_1 \cup M_2]$
of $Q_d$ incuded by $M_1$ and $M_2$ is a collection of bicolored $4$-cycles.
Now, since $M_1$ is heavy and $M_2$ is light, there is a $4$-cycle $C$
in $Q_d[M_1 \cup M_2]$ such that by interchanging colors on $C$
we obtain a coloring $f_1$ where the color class
$f^{-1}_1(1)$ contains at least one less edge that is colored under
$\varphi$ and $f^{-1}_1(2)$ contains at least one more edge that
is colored under $\varphi$.

We shall apply this procedure iteratively and repeatedly select previously unused
edges of a light color class that are not colored under $\varphi$
(where {\em unused} means that the edges have not been involved in any
interchanges performed by the algorithm before),
together with previously unused edges
from a heavy color class,
at least one of which is colored under $\varphi$,
which together form a bicolored $4$-cycle,
and then interchange colors on this
$4$-cycle. Thus we shall construct a sequence of colorings
$f_1,\dots, f_q$, where $f_{i+1}$ is obtained from $f_i$ by interchanging
colors on a bicolored $4$-cycle, and
$f_q$ is the required coloring $f'$ where every color class contains at most
$\frac{7d}{8}$ $\varphi$-colored edges.
Note that since $Q_d$ contains at most $d^2/8$ $\varphi$-colored edges,
$q \leq d^2 /8$.

We now give a brief counting argument for showing that as long
as there is a heavy color class, there is a $4$-cycle in the
current coloring $f_i$ so that after interchanging colors
on this $4$-cycle, the obtained coloring $f_{i+1}$ contains fewer or equally
many heavy color classes, but in the latter case one heavy color class 
contains fewer
$\varphi$-colored edges.

Suppose that $Q_d$ initially contains $k$ heavy color classes
under the coloring $f$, where
$k \leq d$, and that exactly $\alpha(d)$ $\varphi$-colored edges are
not contained in the heavy $k$ color classes in $Q_d$,
where $\alpha(d) \leq d^2/8$ is some function of $d$.
Consider a color class $M$ that is heavy under $f_i$.
Suppose that $M$ (initially) contains $\beta(d)$ $\varphi$-colored edges,
where $\beta(d) \leq d^2/8$ is some function of $d$.
By Lemma \ref{lem:4cycle}, every edge in $Q_d$ is contained in $d-1$
$2$-colored $4$-cycles under $f$, so initially there are at least
$$\frac{(d-k)}{2}\beta(d) -  \alpha(d)$$
$4$-cycles containing edges from $M$ that may be used by the algorithm,
because every $\varphi$-colored edge of a heavy color class is contained in
$(d-k)$ $4$-cycles, where two edges are in a light color class,
and up to $\alpha(d)$
such cycles are unavailable since they contain a $\varphi$-colored
edge of a light color class.

Now, after performing some steps of this algorithm
we might have used edges from some of these cycles in some steps of 
the algorithm. Suppose that the algorithm have used 
\begin{itemize}

\item
$s$ $4$-cycles $C$
with two edges from $M$, such that both edges from $M$ in $C$ are
$\varphi$-colored, and 

\item $r$ $4$-cycles $C$
with two edges from $M$, such that one edge from $M$ in $C$ is
$\varphi$-colored.

\end{itemize}
Moreover, since $Q_d$ contains at most $d^2/8$ $\varphi$-precolored edges
and a light color class contains at most $\frac{7d}{8}-2$ $\varphi$-precolored
edges, we might be unable to use $4$-cycles with edges from at most
$$\frac{2\left(\frac{d^2}{8}-\alpha(d)\right)+ \alpha(d)}
{\frac{7d}{8}-1} \leq \frac{32d}{111}$$
of the $d-k$ initially light color classes,
because such a color class contains at least ${\frac{7d}{8}-1}$
$\varphi$-colored edges after performing some steps of the algorithm,
and, in light of Lemma \ref{lem:onecolor}, we may assume that
$d \geq 16$.

Furthermore,
%there are at most $d \frac{d}{6}$ $\varphi$-colored edges
%in $Q_d$ and 
since
every $4$-cycle that has been used by the algorithm
contains two edges from a light color class, at most 
$$2\left(\frac{d^2}{8} - \alpha(d)-
\frac{7}{8}dk \right)$$
$4$-cycles $C$ are unavailable because $C$ contains an edge from a light color class
that was used previously in another $4$-cycle.
Similarly, for every edge from $M$, there are at most
$b(d)$ parallel edges that are $\varphi$-colored, so at most
$$2b(d) (s+r)$$
$4$-cycles $C$ are unavailable because it contains an edge from $M$ 
that was used previously in another $4$-cycle by the algorithm.
Consequently, if
\begin{align}
\label{majoreq}
\frac{\left(d-k-\frac{32d}{111}\right)}{2}(\beta(d) -2s-r ) - \alpha(d)
- 2\left(\frac{d^2}{8} - \alpha(d)-
\frac{7}{8}dk\right)
-2b(d) (s+r) \geq 1
\end{align}
then we can perform all the necessary steps in the algorithm and thus
the required coloring $f'$ exists.

Now, since $M$ is a heavy color class, 
$\beta(d) \geq \frac{7d}{8}$, and by definition $\beta(d) \leq a(d)$. 
Moreover,
since each color class may contain up to $\frac{7d}{8}$ $\varphi$-colored
edges when the algorithm terminates, 
we have that $2s+r \leq \beta(d) - \frac{7d}{8}$, thus
\eqref{majoreq} holds if

\begin{align}
\label{eq2}
\frac{109d^2}{1776} - k \frac{7d}{8} + \alpha(d) + \frac{7}{4}dk
- 2b(d)\left(a(d) - \frac{7d}{8}\right) \geq 1 
\end{align}

Now, by assumption, 
$\frac{109}{1776} d^2 - 2 b(d)\left(a(d) - \frac{7d}{8}\right)\geq 0$,
so \eqref{majoreq} does indeed hold.

\bigskip
Let us now prove part (ii). The proof of this part is similar to the proof of
part (i). We shall prove that we can perform all the necessary steps in the
algorithm described above, and choose each $4$-cycle $C$ that 
is used by the algorithm
in such a way that for each of the edges of $C$ 
that belongs to a heavy color class, there are at most 
$\frac{d}{34C_1}$ parallel unused edges
that are $\varphi$-colored.
Part (ii) of the theorem then holds if \eqref{majoreq} is valid
under the assumptions that
$a(d) =C_1d$ and $b(d) =\frac{d}{34C_1}$.
Since $\frac{109}{1776} > \frac{1}{17}$, this, in turn, 
follows from the fact that \eqref{eq2} holds, given that
$a(d) =C_1d$ and $b(d) =\frac{d}{34C_1}$.

Our task is thus to prove that in each step of the algorithm, we can select
a $4$-cycle so that each of the edges from the heavy color class are parallel
with at most $\frac{d}{34C_1}$ unused $\varphi$-colored edges.

So suppose that some steps of the algorithm have been performed
and we have selected some $4$-cycles satisfying this condition.
Then, since \eqref{majoreq} holds, there is some $4$-cycle
$C=uvxyu$ that is edge-disjoint from all previously considered $4$-cycles
and such that $uv$ and $xy$ are edges from some heavy color class, at least
one of which is $\varphi$-colored, and the edges $vx$ and $yu$ are not 
$\varphi$-colored and lie in a color class that is light
under the current coloring $f_i$. Suppose that one of the
edges $uv$ and $xy$, $uv$ say,
are parallel with at least $\frac{d}{34C_1}$ unused  $\varphi$-colored edges. 
Denote by $M_1$ the dimensional matching containing $uv$ and
consider the set
$E' \subseteq M_1$ of all these $\varphi$-colored edges
that are parallel to $uv$.
At most $\frac{d}{34(C_1+2)}$ of the edges in $E'$ are parallel with at least 
$\frac{d}{34C_1}$ $\varphi$-colored edges
edges, because any edge (except $uv$) that is parallel with an edge from $E'$ is
parallel with at most one other edge from $E'$ and
$\frac{d}{34C_1}\frac{d}{34(C_1+2)}\frac{1}{2} = \frac{d^2}{C_2}$, and $Q_d$
contains altogether at most $\frac{d^2}{C_2}$ $\varphi$-colored edges.

Let $E'' \subseteq E'$ be the set of edges that are
parallel with $uv$ and which
are parallel with at most $\frac{1}{34C_1}$ $\varphi$-colored edges. Then
$$|E''| \geq \frac{d}{34C_1} - \frac{d}{34(C_1+2)} = \frac{4d}{C_2}.$$

Next, we shall estimate the number of $4$-cycles with unused edges,
that contains exactly one edge
from $E''$, and two edges from a light color class,
and satisfying that two cycles containing different edges from $E''$
are disjoint.
Now, since $|E''| \geq \frac{4d}{C_2}$ and any edge that is parallel with an
edge from $E'$ are contained in at most $2$ such $4$-cycles,
arguing as above, we deduce that
there are at least
$$\frac{d-k-\frac{d}{C_2-1}}{2}\frac{4d}{C_2} - 
\alpha(d) - \left(d\frac{d}{C_2}-\alpha(d) 
- kd \frac{C_2-1}{C_2}\right) \geq \frac{d^2(C_2-3)}{C_2(C_2-1)}$$
such unused cycles. Denote the set of all such cycles by $\mathcal{C}$.

By construction, all the edges of $E''$ that are in cycles in $\mathcal{C}$
are parallel with at most $\frac{d}{34C_1}$ unused $\varphi$-colored edges. 
We shall
prove that this holds for both edges of $M_1$ in at least one of the
cycles of $\mathcal{C}$.

Consider a cycle $C = abcda \in \mathcal{C}$, where $ab \in E''$,
$cd \in M_1 \setminus E''$, the edges $ua$ and $bv$ are contained
in the dimensional matching $M_{i}$, and the edges $bc$ and $ad$
are contained in the dimensional matching $M_j$.
Now, if $cd$ is parallel with at least $\frac{d}{34C_1}$ 
unused $\varphi$-precolored edges,
then there are at least $\frac{d}{34C_1}-2$ such edges $c'd' \in M_1$, where
$cc' \in E(Q_d)$ and $dd' \in E(Q_d)$, such that $cc', dd' \notin M_i$.
Suppose, for instance, $cc',dd' \in M_k$, where $k \neq i$.
Then, since $i \neq k$, and there are six permutations of the
matchings $M_i, M_j,M_k$,
it follows from Proposition \ref{prop:HavelMoravek},
that there are at most $5$ other cycles from $\mathcal{C}$ that
contain an edge which is parallell with $c'd'$.
Summing up, we conclude that if all cycles in $\mathcal{C}$ contains
an edge from $M_1$ that is parallel with at least $\frac{d}{34C_1}$ 
unused $\varphi$-colored
edges, then $Q_d$ contains at least
$$\frac{d^2(C_2-3)}{C_2(C_2-1)}\left(\frac{d}{34C_1} -2\right) \frac{1}{6}$$
$\varphi$-colored edges. However, 
by Lemma \ref{lem:onecolor}, we may assume that $d \geq 2C_2$, so
this is not possible since $Q_d$ contains
at most $\frac{d^2}{C_2}$ precolored edges.
We conclude that at least one cycle in $\mathcal{C}$ satisfies that every 
edge from $M_1$ is parallel with at most $\frac{d}{34C_1}$ 
other $\varphi$-precolored edges.
Consequently, we can perform all the necessary steps in the algorithm
to obtain the required coloring $f'$.
\end{proof}

It is trivial that Conjecture \ref{conj:general} is true
in the case when only one color appears in the coloring that is
to be avoided; the case of two involved colors is also
straightforward.
We give a short argument showing
that Conjecture \ref{conj:general} holds
in the case when the partial coloring uses at most three colors.

\begin{proposition}
\label{prop:4col}
	If $\varphi$ is a partial edge coloring of $Q_d$ with at most three colors
	and every color appears on at most $d-2$ edges,
	then $\varphi$ is avoidable.
\end{proposition}
\begin{proof}
	It is a simple exercise to show that the result holds in the case when $d=3$.
	Thus, we may assume that $d \geq 4$ and that exactly three
	colors appear in the partial coloring $\varphi$.
	
	Let $f$ be standard edge coloring of $Q_d$, 
	and consider the bipartite graph $B(f)$
	with parts consisting of the
	color classes $C(f)$ of $f$ and the colors $\{1,\dots,n\}$
	used in $\varphi$, and
	where an edge appears between a color $i$ of $\varphi$ and a color
	class $M_j$ of $f$ if and only if no edge of $M_j$ is colored $i$ 
	under $\varphi$.
	
	As in the proof of the preceding theorem,
	if there is a perfect matching in $B(f)$, then the coloring $\varphi$
	is avoidable, so suppose that this is not the case. Then
	there is an anti-Hall set $S \subseteq C(f)$, that is,
	a set $S \subseteq C(f)$, such that $|N(S)| < |S|$.
	Our goal is to prove that there is a coloring $f'$ that
	can be obtained from $f$ by interchanging colors on some $4$-cycles,
	so that in the bipartite graph $B(f')$, defined as above, there
	is a perfect matching.
	
	Now, if $S$ is a anti-Hall set, then
	since every color in $\varphi$ appears at most $d-2$ times,
	$|S| \leq d-2$. On the other hand, since at most $3$ colors
	appear in the coloring $\varphi$, $|N(S)| \geq d-3$, so
	$|S| \geq d-2$; consequently, $|S| = d-2$, that is,
	every dimensional matching in $S$ contains edges of
	all three colors under $\varphi$, and thus there are two 
	dimensional matchings in $Q_d$
	where no edges are colored under $\varphi$. 
	Without loss of generality, we assume that $M_1$ is a 
	dimensional matching
	with color $1$ under $f$ that is in $S$. 
	If $d \geq 5$, then we pick a dimensional
	matching, $M_d$, with color $d$ under $f$, say, not contained in the
	set $S$. Now, since $M_d$ contains no $\varphi$-colored edges
	and $M_1$ contains $d-2$ such edges, there is a $4$-cycle in 
	the edge-induced subgraph $Q_d[M_1 \cup M_d]$
	containing at least one $\varphi$-colored edge. Since $d \geq 5$, 
	by interchanging colors on this $4$-cycle, we obtain the
	required coloring $f'$.
	
	It remains to consider the case when $d=4$. Let $M_1$ be a dimensional
	matching in $S$; then $M_1$ contains exactly one edge $\varphi$-colored
	$i$, for $i=1,2,3$. The graph $Q_d -M_1$, consisting of two copies of the
	graph $Q_3$, thus contains exactly one edge colored $i$, $i=1,2,3$; so
	by our initial observation, there is a proper $3$-edge coloring of 
	$Q_d-M_1$ that avoids the restriction of $\varphi$ to $Q_d-M_1$.
	Now, by assigning color $4$ to $M_1$, we obtain a proper $d$-edge coloring
	of $Q_d$ that avoids $\varphi$.
\end{proof}

\begin{remark}
	We remark that by using the same strategy it is straightforward to prove
	a version of the preceding result with four instead of three colors, provided
	that $d \geq 5$; indeed, the only essential difference is that one has to consider
	two different cases on the size of the anti-Hall set, namely, when it has size 
	$d-2$ and $d-3$, respectively.
	However, for the case when $d=4$, the only proof we have proceeds by long and
	detailed case analysis, so
	we abstain from giving the details in the case when the coloring to be avoided
	contains four different colors.
\end{remark}

As a final observation of this section, let us consider the case when all precolored
edges lie in a hypercube of dimension $d-1$ contained in a $d$-dimensional 
hypercube.

The following was first conjectured in \cite{Johansson}.

\begin{proposition}
	If $\varphi$ is a partial $d$-edge coloring of the hypercube $Q_d$ ($d \geq 2$),
	where all colored edges
	lie in a subgraph that is isomorphic to $Q_{d-1}$, then $\varphi$ is avoidable.
\end{proposition}

\begin{proof}
	Let $H_1$ be a subgraph of $Q_{d}$ that is isomorphic to $Q_{d-1}$ and contains
	all precolored edges. Then $Q_d$ consists of the two copies $H_1$ and $H_2$
	of $Q_{d-1}$ and a dimensional matching $M$ joining vertices of $H_1$ and $H_2$.

	We define a list assignment $L$ for $H_1$ by setting 
	$L(e) = \{1,\dots, d\} \setminus \{\varphi(e)\}$, 
	for every edge $e \in E(H_1)$,
	where
	we assume that $\varphi(e) = \emptyset$ if $e$ is not colored under $\varphi$.
	By Galvin's Theorem \ref{th:Galvin}, there is a proper $d$-edge coloring
	of $H_1$ with colors from the lists. Since $H_1$ and $H_2$ are isomorphic,
	this also yields a corresponding $d$-edge coloring of $H_2$.
	By coloring all edges of $M$ by the unique color in $\{1,\dots,d\}$ missing
	at its endpoints, we obtain a proper $d$-edge coloring of $Q_d$ which avoids
	$\varphi$.
\end{proof}

%%%%%%%%%%%%

\section{Avoiding partial proper edge colorings}

In \cite{Johansson}, Johansson presented a complete list of minimal unavoidable
partial $3$-edge colorings of $Q_3$, where {\em minimal} means that 
removing a color from any colored edge yields an avoidable edge coloring;
the list is {\em complete} in the sense that it contains all such colorings up to
permuting colors and/or applying graph automorphisms.
There are 29 such configurations,
and we refer to \cite{Johansson} for a comprehensive list of all such colorings.
Let us here just remark that, based on this list of minimal unavoidable partial
edge colorings, it seems to be a difficult task to characterize the family
of unavoidable partial edge colorings of $Q_d$ for general $d$. 
Note further that a similar
investigation for complete bipartite graphs was pursued
in \cite{MarkstromOhman}.

Here, we shall focus on the unavoidable partial proper $3$-edge colorings
of $Q_3$.
As explained in \cite{Johansson}, there are six such minimal configurations.

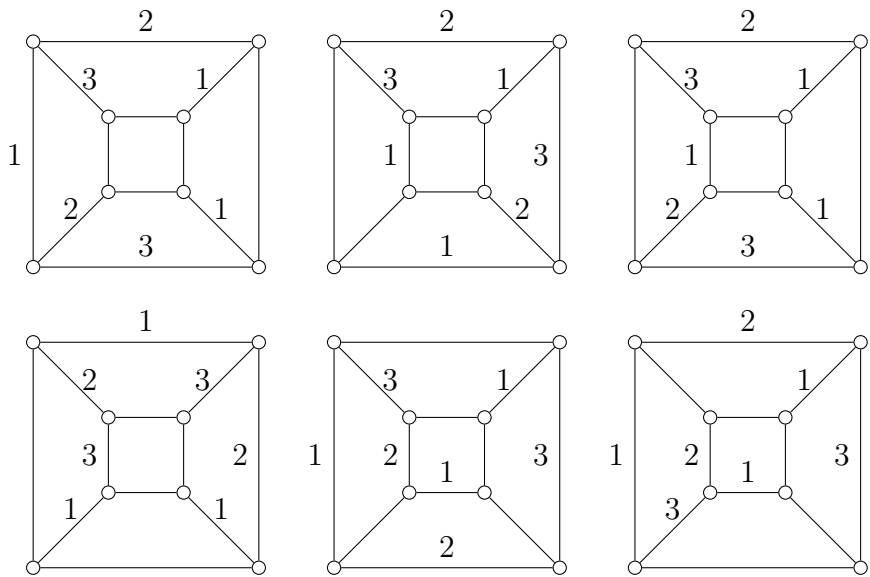
\begin{figure} [H]
	\begin{center}

	\begin{tikzpicture}[scale=0.5]
	\tikzset{vertex/.style = {shape=circle,draw,
												inner sep=0pt, minimum width=5pt}}
	\tikzset{edge/.style = {-,> = latex'}}
	% vertices
	%\node at (1,4) {}
	
	\node[vertex] (x_11) at  (-4, 0) {};
	\node[vertex] (x_12) at  (-4, 6) {};
	\node[vertex] (x_13) at  (-2, 2) {};
	\node[vertex] (x_14) at  (-2, 4) {};
	\node[vertex] (x_21) at  (0, 2) {};
	\node[vertex] (x_22) at  (0, 4) {};
	\node[vertex] (x_23) at  (2, 0) {};
	\node[vertex] (x_24) at  (2, 6) {};
	
	%edges
	\draw[edge] (x_11) to node [midway,left] {$1$} (x_12);
		\draw[edge] (x_11) to node [midway,above] {$2$} (x_13);
	   \draw[edge] (x_11) to node [midway,above] {$3$} (x_23);
   		\draw[edge] (x_12) to node [midway,right] {$3$} (x_14);
		\draw[edge] (x_12) to node [midway,above] {$2$} (x_24);
		\draw[edge] (x_13) to  (x_14);
		\draw[edge] (x_13) to  (x_21);
		\draw[edge] (x_14) to  (x_22);
		\draw[edge] (x_21) to  (x_22);
		\draw[edge] (x_21) to node [midway,above] {$1$} (x_23);
		\draw[edge] (x_22) to node [midway,left] {$1$} (x_24);
	\draw[edge] (x_23) to  (x_24);

	%%%%%%%%%%%%%%%%%%
	
		\node[vertex] (a_11) at  (4, 0) {};
	\node[vertex] (a_12) at  (4, 6) {};
	\node[vertex] (a_13) at  (6, 2) {};
	\node[vertex] (a_14) at  (6, 4) {};
	\node[vertex] (a_21) at  (8, 2) {};
	\node[vertex] (a_22) at  (8, 4) {};
	\node[vertex] (a_23) at  (10, 0) {};
	\node[vertex] (a_24) at  (10, 6) {};
	
	%edges
	\draw[edge] (a_11) to (a_12);
		\draw[edge] (a_11) to  (a_13);
	  \draw[edge] (a_11) to node [midway,above] {$1$} (a_23);
   	\draw[edge] (a_12) to node [midway,right] {$3$} (a_14);
		\draw[edge] (a_12) to node [midway,above] {$2$} (a_24);
		\draw[edge] (a_13) to node [midway,left] {$1$} (a_14);
		\draw[edge] (a_13) to  (a_21);
		\draw[edge] (a_14) to  (a_22);
		\draw[edge] (a_21) to  (a_22);
		\draw[edge] (a_21) to node [midway,above] {$2$} (a_23);
		\draw[edge] (a_22) to node [midway,left] {$1$} (a_24);
		\draw[edge] (a_23) to node [midway,left] {$3$} (a_24);

	%%%%%%%%%%%%%%%%%

	\node[vertex] (b_11) at  (12, 0) {};
	\node[vertex] (b_12) at  (12, 6) {};
	\node[vertex] (b_13) at  (14, 2) {};
	\node[vertex] (b_14) at  (14, 4) {};
	\node[vertex] (b_21) at  (16, 2) {};
	\node[vertex] (b_22) at  (16, 4) {};
	\node[vertex] (b_23) at  (18, 0) {};
	\node[vertex] (b_24) at  (18, 6) {};
	
	%edges
	\draw[edge] (b_11) to (b_12);
		\draw[edge] (b_11) to node [midway,above] {$2$} (b_13);
	  \draw[edge] (b_11) to node [midway,above] {$3$} (b_23);
   	\draw[edge] (b_12) to node [midway,right] {$3$} (b_14);
		\draw[edge] (b_12) to node [midway,above] {$2$} (b_24);
		\draw[edge] (b_13) to node [midway,left] {$1$} (b_14);
		\draw[edge] (b_13) to  (b_21);
		\draw[edge] (b_14) to  (b_22);
		\draw[edge] (b_21) to  (b_22);
		\draw[edge] (b_21) to node [midway,above] {$1$} (b_23);
		\draw[edge] (b_22) to node [midway,left] {$1$} (b_24);
		\draw[edge] (b_23) to (b_24);
	
	%%%%%%%%%%%%%%%%%%%%%
	
	\node[vertex] (c_11) at  (-4, -8) {};
	\node[vertex] (c_12) at  (-4, -2) {};
	\node[vertex] (c_13) at  (-2, -6) {};
	\node[vertex] (c_14) at  (-2, -4) {};
	\node[vertex] (c_21) at  (0, -6) {};
	\node[vertex] (c_22) at  (0, -4) {};
	\node[vertex] (c_23) at  (2, -8) {};
	\node[vertex] (c_24) at  (2, -2) {};
	
	%edges
	\draw[edge] (c_11) to (c_12);
		\draw[edge] (c_11) to node [midway,above] {$1$} (c_13);
	   \draw[edge] (c_11) to (c_23);
   		\draw[edge] (c_12) to node [midway,right] {$2$} (c_14);
		\draw[edge] (c_12) to node [midway,above] {$1$} (c_24);
		\draw[edge] (c_13) to node [midway,left] {$3$} (c_14);
		\draw[edge] (c_13) to  (c_21);
		\draw[edge] (c_14) to  (c_22);
		\draw[edge] (c_21) to  (c_22);
		\draw[edge] (c_21) to node [midway,above] {$1$} (c_23);
		\draw[edge] (c_22) to node [midway,left] {$3$} (c_24);
		\draw[edge] (c_23) to node [midway,left] {$2$} (c_24);

	%%%%%%%%%%%%%%%%%%
	
		\node[vertex] (d_11) at  (4, -8) {};
	\node[vertex] (d_12) at  (4, -2) {};
	\node[vertex] (d_13) at  (6, -6) {};
	\node[vertex] (d_14) at  (6, -4) {};
	\node[vertex] (d_21) at  (8, -6) {};
	\node[vertex] (d_22) at  (8, -4) {};
	\node[vertex] (d_23) at  (10, -8) {};
	\node[vertex] (d_24) at  (10, -2) {};
	
	%edges
	\draw[edge] (d_11) to node [midway,left] {$1$} (d_12);
		\draw[edge] (d_11) to  (d_13);
	  \draw[edge] (d_11) to node [midway,above] {$2$} (d_23);
   	\draw[edge] (d_12) to node [midway,right] {$3$} (d_14);
		\draw[edge] (d_12) to (d_24);
		\draw[edge] (d_13) to node [midway,left] {$2$} (d_14);
		\draw[edge] (d_13) to  node [midway,above] {$1$}(d_21);
		\draw[edge] (d_14) to  (d_22);
		\draw[edge] (d_21) to  (d_22);
		\draw[edge] (d_21) to (d_23);
		\draw[edge] (d_22) to node [midway,left] {$1$} (d_24);
		\draw[edge] (d_23) to node [midway,left] {$3$} (d_24);

	%%%%%%%%%%%%%%%%%

	\node[vertex] (e_11) at  (12, -8) {};
	\node[vertex] (e_12) at  (12, -2) {};
	\node[vertex] (e_13) at  (14, -6) {};
	\node[vertex] (e_14) at  (14, -4) {};
	\node[vertex] (e_21) at  (16, -6) {};
	\node[vertex] (e_22) at  (16, -4) {};
	\node[vertex] (e_23) at  (18, -8) {};
	\node[vertex] (e_24) at  (18, -2) {};
	
	%edges
	\draw[edge] (e_11) to node [midway,left] {$1$} (e_12);
		\draw[edge] (e_11) to node [midway,above] {$3$} (e_13);
	  \draw[edge] (e_11) to (e_23);
   	\draw[edge] (e_12) to (e_14);
		\draw[edge] (e_12) to node [midway,above] {$2$} (e_24);
		\draw[edge] (e_13) to node [midway,left] {$2$} (e_14);
		\draw[edge] (e_13) to node [midway,above] {$1$} (e_21);
		\draw[edge] (e_14) to  (e_22);
		\draw[edge] (e_21) to  (e_22);
		\draw[edge] (e_21) to (e_23);
		\draw[edge] (e_22) to node [midway,left] {$1$} (e_24);
		\draw[edge] (e_23) to node [midway,left] {$3$} (e_24);

\end{tikzpicture}
\end{center}
\caption{Minimal unavoidable partial proper $3$-edge colorings of $Q_3$.}
	\label{fig:Q3prop}
\end{figure}

\begin{proposition}
\label{prop:Q3}
	The partial edge colorings of $Q_3$ in Figure \ref{fig:Q3prop} constitute
	a complete list of minimal unavoidable partial proper $3$-edge colorings of $Q_3$.
\end{proposition}

The proof of this proposition is by an exhaustive computer search; we refer
to \cite{Johansson} for details.
		
As in the non-proper case,
based on this list of minimal unavoidable partial proper $3$-edge colorings of $Q_3$,
it seems difficult to make any specific conjecture as to whether it is possible to 
characterize the minimal unavoidable partial proper $d$-edge colorings of $Q_d$
for general $d$.
It is, however, easy to construct infinite families of 
minimal unavoidable partial (non-proper) $3$-edge colorings
of hypercubes; for the case when the coloring is required to be proper, 
this problem appears to be more difficult; 
in fact, we are interested in whether
the following might be true:

\begin{problem}
\label{prob:proper}
	Is there an integer $d_0 \geq 0$ such that every partial
	proper $d$-edge coloring of
	$Q_d$ is avoidable if $d \geq d_0$?
\end{problem}

As mentioned in the introduction above, for the balanced
complete bipartite graphs
the answer to the corresponding question
is positive and it suffices to require that the
graph has at least $8$ vertices \cite{Cavenagh, ChetwyndRhodes, Ohman}.

%Let us also note that for $Q_2$, there are apparently
%two minimal unavoidable partial $2$-edge colorings, one of which is proper.
%and for $Q_1$ there is one such configuration.

\bigskip

Next, we shall deduce some general consequences
of Proposition \ref{prop:Q3}. We begin by considering the
special case of Problem \ref{prob:proper} when all
colored edges are contained in a matching. We shall need the following lemmas,
which are immediate from Proposition \ref{prop:Q3}.

\begin{lemma}
\label{lem:match}
	If $\varphi$ is a partial $3$-edge coloring $Q_3$ where all colored
	edges are contained in a matching,
	then $\varphi$ is avoidable.
\end{lemma}

We note that an analogous statement does not hold for $Q_2$, since the partial coloring
where two non-adjacent edges of $Q_2$ are colored by $1$ and $2$, respectively, 
is unavoidable.

\begin{lemma}
\label{lem:2Match}
	If $\varphi$ is a partial proper $3$-edge coloring of $Q_3$ where all
	colored edges are contained in two dimensional matchings, then $\varphi$
	is avoidable.
\end{lemma}

%We conjecture the following.
%
%\begin{conjecture}
%	If $\varphi$ is a partial $d$-coloring of $Q_d$ ($d \geq 3$) where all colored
%	edges are contained in a matching, then $\varphi$ is avoidable.
%\end{conjecture}

%We prove this conjecture when $d$ is divisible by $3$.

\begin{corollary}
\label{cor1}
	If $d=3k$ and $\varphi$ is a partial $d$-edge coloring of $Q_d$ where all 
	colored edges
	are contained in a matching, then $\varphi$ is avoidable.
\end{corollary}
\begin{proof}
	Let $M_1, \dots, M_d$ be the dimensional matchings in $Q_d$. For $i=1,\dots,k$,
	let $H_i$ be the subgraph of $Q_d$ induced by 
	$M_{3i-2} \cup M_{3i-1} \cup M_{3i}$. By Lemma \ref{lem:DimMathInduce},
	each $H_i$ is a collection of disjoint $3$-dimensional hypercubes.
	
	Now, by Lemma \ref{lem:match}, there is a proper edge coloring
	of $H_i$ using colors $3i-2, 3i-1, 3i$ that avoids the restriction of
	$\varphi$ to $H_i$, for $i=1,\dots,k$. Combining such colorings
	yields a proper $d$-edge coloring of $Q_d$ that avoids $\varphi$.
\end{proof}

If we insist that all precolored edges are contained in a bounded number of
dimensional matchings, then we obtain another family of avoidable partial
(not necessarily proper) $d$-edge colorings of $Q_d$.

\begin{corollary}
\label{cor2}
		If $\varphi$ is a partial $d$-edge coloring of $Q_d$ where all
		colored edges are contained in $\lfloor d/3 \rfloor$ dimensional matchings,
		then $\varphi$ is avoidable.
\end{corollary}
\begin{proof}
	Suppose that $M_1,\dots, M_a$ are the dimensional matchings
	that contain edges that are colored under $\varphi$, where
	$a=\lfloor d/3 \rfloor$.
	As in the proof of the preceding corollary, we decompose $Q_d$
	into $a=\lfloor d/3 \rfloor$ subgraphs $H_1,\dots, H_a$ consisting of
	$3$-dimensional hypercubes, and possibly one subgraph $H_{a+1}$ that consists
	of disjoint copies of $1$- or $2$-dimensional hypercubes. Moreover,
	without loss of generality we assume that  $M_i$ is contained in $H_i$,
	$i=1,\dots,d$. The result now follows from Lemma \ref{lem:match}
	as in the proof of Corollary \ref{cor1}.
\end{proof}

If we require that the partial coloring is proper, 
then we can allow up to $2\lfloor d/3 \rfloor$
dimensional matchings in $Q_d$ containing colored edges, while still
being able to avoid the partial coloring.

\begin{corollary}
\label{cor3}
		If $\varphi$ is a  partial proper $d$-edge coloring of $Q_d$ where all
		colored edges are contained in $2\lfloor d/3 \rfloor$ dimensional matchings,
		then $\varphi$ is avoidable.
\end{corollary}

The only difference in the proof of Corollary \ref{cor3} 
compared to the proof of Corollary \ref{cor2} is that we use Lemma \ref{lem:2Match}
in place of Lemma \ref{lem:match}; we omit the details.

A weaker and perhaps more tractable version of Problem
\ref{prob:proper} is obtained by requiring that every color class in the 
partial edge coloring
to be avoided is an induced matching.

\begin{conjecture}
\label{conj:induced}
	If $d \geq 3$ and $\varphi$ is a partial $d$-edge coloring of 
	$Q_d$ where every color class
	is an induced matching, then $\varphi$ is avoidable.
\end{conjecture}

Using Proposition \ref{prop:Q3} and proceeding as in the proofs of the preceding
Corollaries, we can prove the following stronger version of Conjecture
\ref{conj:induced} in the case when $d$ is divisible by $3$.

\begin{corollary}
	If $d=3k$ and $\varphi$ is a partial proper $d$-edge coloring
	of $Q_d$, where every precolored edge is at distance $1$ from at most one other
	edge with the same color, then $\varphi$ is avoidable.
\end{corollary}

Finally, we shall prove that Conjecture
\ref{conj:general} is true in the case when the partial edge coloring is proper.
We shall need the following easy lemma.

\begin{lemma}
\label{lem:Q2}
	If $\varphi$ is a partial proper $2$-edge coloring of $Q_2$, then $\varphi$
	is avoidable unless two non-adjacent edges are colored by different colors.
\end{lemma}

\begin{theorem}
\label{th:prop}
	If $\varphi$ is a partial proper $d$-edge coloring of $Q_d$ where
	every color appears on at most $d-2$ edges, then $\varphi$ is
	avoidable.
\end{theorem}

\begin{proof}
If $d \leq 3$, then the theorem trivially holds by Proposition \ref{prop:Q3}.
If $d=4$, then consider a dimensional matching $M$ not containing edges of all colors
under $\varphi$; such a matching exists, since every color appears on at most $d-2$ edges
under $\varphi$. Suppose, for example, that color $4$ does not appear on an edge of $M$.

The graph $Q_d-M$ consists of two disjoint copies $H_1$ and $H_2$ of $Q_3$. 
Moreover,
since every color appears on at most two edges, it follows from
Proposition \ref{prop:Q3} that there are proper $3$-edge colorings (using colors $1,2,3$)
of $H_1$ and $H_2$ that avoid the restrictions of $\varphi$ to $H_1$ and $H_2$, respectively.
By coloring all edges of $M$ by color $4$, we obtain a proper 
$4$-edge coloring of $Q_d$ that avoids $\varphi$.

Next, we consider the case when $d \geq 5$.
We shall consider two main cases, namely when $d=2k$, and when
$d=2k+1$.

Let us first consider the case when $d=2k \geq 6$. Denote by
$M_1,\dots, M_{2k}$ the dimensional matchings of $Q_d$ and consider the subgraphs
$H_1, \dots, H_k$ of $Q_d$, where $H_i$ is the subgraph induced by 
$M_{2i-1} \cup M_{2i}$. We shall partition the colors in $\{1,\dots, 2k\}$
into $2$-subsets $A_1, \dots, A_k$ and use the colors in $A_i$ for a proper
edge coloring of $H_i$ that avoids the restriction of $\varphi$ to $H_i$.
Combining all these colorings yields a $d$-edge coloring of $Q_d$ which avoids $\varphi$.

In total, there are $\frac{(2k)!}{(2!)^k}$ ordered
partitions of $\{1,\dots, 2k\}$
into $k$ $2$-subsets.
Now, some of these partitions $A_1 \cup \dots \cup A_k$ are {\em forbidden} 
in the sense that for some $i$ there is no proper edge coloring
of $H_i$ using colors from $A_i$ that avoids the restriction of $\varphi$ to $H_i$.
If a copy of $Q_2$, contained in some subgraph $H_i$, contains at most two
$\varphi$-colored edges, then by Lemma \ref{lem:Q2} at most 
$\frac{(2k-2)!}{(2!)^{k-1}}$
partitions of $\{1,\dots, 2k\}$ are forbidden due to this coloring of $Q_2$;
similarly, if at most four edges of $Q_2$ are colored under $\varphi$, then
at most $2 \frac{(2k-2)!}{(2!)^{k-1}}$ partitions are forbidden.

Now, since $Q_d$ contains at most $d(d-2)$ $\varphi$-colored edges,
at most $\frac{d(d-2)}{2} \frac{(2k-2)!}{(2!)^{k-1}}$ partitions of
$\{1,\dots, 2k\}$ are forbidden due to the condition that the resulting $d$-edge coloring
of $Q_d$ should avoid $\varphi$. Thus, if
$$\frac{(2k)!}{(2!)^k} - \frac{d(d-2)}{2}\frac{(2k-2)!}{(2!)^{k-1}}>0,$$
then there is a non-forbidden partition of $\{1,\dots, 2k\}$. Since this
inequality holds for any $k \geq 1$, the desired result follows.

Let us now consider the case when $d=2k+1$. 
The argument here is similar to the one given above. We partition
$Q_d$ into the subgraphs $H_1, \dots, H_k$, where $H_i$
is induced by the dimensional matchings $M_{2i-1} \cup M_{2i}$, $i=1,\dots, k-1$,
and $H_k$ is induced by $M_{2k-1}, M_{2k}, M_{2k+1}$. We now seek a partition
of $\{1,\dots, 2k+1\}$ into sets $A_1 \cup \dots \cup A_{k}$, 
where $|A_i|=2$, $i=1,\dots k-1$,
and $|A_{k}| =3$, and corresponding proper edge colorings of $H_1,\dots, H_k$, where
a coloring of $H_i$ uses colors from $A_i$. 

In total, there are $\frac{(2k+1)!}{(2!)^{k-1}3!}$ such ordered partitions of
$\{1,\dots, 2k+1\}$.
As before, some of these partitions are forbidden due to the fact
the resulting edge coloring should avoid $\varphi$.
We shall need the following claim.

\begin{claim}
\label{cl:Q3}
		Let $\varphi$ be a partial edge coloring of a copy $H$ of the $3$-dimensional
		hypercube $Q_3$ contained in $H_k$. Let $s(a)$ be the largest 
		number of partitions
		of $\{1,\dots, 2k+1\}$ that are forbidden due to the restriction of
		$\varphi$ to $H$ being unavoidable when $a$ 
		edges of $H$ are colored. Then
		$$
    s(a) \leq 
		\begin{cases}
	      0, & \text{if } a \leq 6,\\
        \frac{(2k-2)!}{(2!)^{k-1}}, & \mbox{if } 7 \leq a \leq 8,\\
        3 \frac{(2k-2)!}{(2!)^{k-1}}, & \mbox{if } a =9, \\
				4 \frac{(2k-2)!}{(2!)^{k-1}}, & \mbox{if } a =10, \\
				6 \frac{(2k-2)!}{(2!)^{k-1}}, & \mbox{if } a =11, \\
				9 \frac{(2k-2)!}{(2!)^{k-1}}, & \mbox{if } a =12. 
			\end{cases}
		$$		
\end{claim}
\begin{proof}
	By Proposition \ref{prop:Q3},
	the partial edge colorings of $Q_3$ in Figure \ref{fig:Q3prop}
	is a complete list of minimal unavoidable partial proper $3$-edge colorings
	of $Q_3$.
	Note that every such partial coloring contains three edges colored $1$,
	two edges colored $2$, and two edges colored $3$. Thus, if $H$
	contains at most six $\varphi$-colored edges, then no partitions
	of $\{1,\dots, 2k+1\}$ are forbidden due to the restriction of $\varphi$
	to $H$ being unavoidable; that is, $s(a)=0$ if $a \leq 6$.
	Similarly, if at most $8$ different $\varphi$-colored edges appear in $H$,
	then at most $\frac{(2k-2)!}{(2!)^{k-1}}$ are forbidden.
	
	If $H$ contains $9$ $\varphi$-colored edges, then at most 
	$3 \frac{(3k-3)!}{(3!)^{k-1}}$ partitions are forbidden, since there could
	be four colors present on edges in $H$, one of which appears on three edges.
	Similarly, it is straightforward that $s(10) \leq 4\frac{(2k-2)!}{(2!)^{k-1}}$, 
	$s(11) \leq 6\frac{(2k-2)!}{(2!)^{k-1}}$, and
	$s(12) \leq 9\frac{(2k-2)!}{(2!)^{k-1}}$.
\end{proof}

Let $b$ be the number of $\varphi$-colored edges that appear
on edges in $H_k$. Then by using the same counting arguments as above
and invoking Claim \ref{cl:Q3}, we deduce that at most 
$$\left(d(d-2)-b\right) \frac{1}{2}\frac{(2k-1)!}{(2!)^{k-2}3!}
+ b \frac{9}{12} \frac{(2k-2)!}{(2!)^{k-1}}$$
partitions of $\{1,\dots, 2k\}$ are forbidden due to the condition that the 
resulting $d$-edge coloring
of $Q_d$ should avoid $\varphi$.
Thus, if
$$\frac{(2k+1)!}{(2!)^{k-1}3!} - 
\left(\left(\frac{d(d-2)}{2}-\frac{b}{2}\right) \frac{(2k-1)!}{(2!)^{k-2}3!}
+ b \frac{9}{12} \frac{(2k-2)!}{(2!)^{k-1}}\right)>0,$$
then there is a non-forbidden partition of $\{1,\dots, 2k+1\}$.
This holds if $k \geq 3$; and if $k=2$, then we can select the
two dimensional matchings contained in $H_1$ to be maximal with respect
to the property of containing $\varphi$-precolored edges. This implies
that $H_k$
contains at most nine $\varphi$-colored edges; that is, $b \leq 9$,
and the required inequality holds.
\end{proof}

\section{Extending and avoiding edge colorings simultaneously}

In \cite{CasselgrenMarkstromPham}, it was proved that any partial
proper coloring of at most $d-1$ edges of $Q_d$ is extendable to a
proper $d$-edge coloring of $Q_d$. Moreover,
it was proved that any partial proper coloring of at most $d$ edges
in $Q_d$ is extendable unless it satisfies one of the following conditions:

\begin{itemize}
	 
		 \item[(C1)]  there is an uncolored edge $uv$ in $Q_d$ such that
		$u$ is incident with edges of $r \leq d$ distinct colors and
		$v$ is incident to $d-r$ edges colored with $d-r$
		other distinct colors (so $uv$
		is adjacent to edges of $d$ distinct colors);
		
		\item[(C2)]  there is a vertex $u$ 
		and a color $c$ such that $u$ is incident with at least one colored
		edge, $u$ is not incident with any edge of color $c$, and every
		uncolored edge incident with $u$ is adjacent to another edge colored $c$;

		\item[(C3)] there is a vertex $u$ and a color $c$
		such that every edge
		incident with $u$ is uncolored
		and every edge incident with $u$ is adjacent to another edge colored $c$;

		\item[(C4)]
		$d=3$ and the three precolored edges use three different colors and
		form a subset of a dimensional matching.

	\end{itemize}
		For $i=1,2,3,4$, we denote by $\mathcal{C}_i$
	the set of all colorings of $Q_d$, $d \geq 1$,
	satisfying the corresponding condition above, and we set
	$\mathcal{C} = \cup \mathcal{C}_i$.
	
	\begin{theorem}
	\label{extQ_d}
	\cite{CasselgrenMarkstromPham}
	If $\varphi$ is a partial proper $d$-edge coloring of at most
	$d$ edges in $Q_d$, then $\varphi$ is extendable
	to a proper $d$-edge coloring of $Q_d$ unless  $\varphi \in \mathcal{C}$.
	\end{theorem}

For $1 \leq k \leq d$, let
$\varphi$ be a proper precoloring
of $d-k$ edges of $Q_d$ and $\psi$ be a partial coloring of
$k$ edges in $Q_d$. Using the preceding theorem, we shall prove that
there is a proper $d$-edge coloring of $Q_d$ that agrees with $\varphi$ and
which avoids $\psi$ unless one of the following
conditions are satisfied:

\begin{itemize}

	\item[(D1)] there is a vertex $v$ such that every edge incident with $v$
	is either $\psi$-colored $c$, $\varphi$-colored by a color distinct from $c$,
	or not colored under $\varphi$ or $\psi$, but adjacent to an edge
	with color $c$ under $\varphi$; or
	
	\item[(D2)] exactly one edge $uv$ is colored under $\psi$
	and for every $i \in \{1,\dots,d\}\setminus \{\psi(uv)\}$ there is an
	edge incident with $u$ or $v$ that is colored $i$ under $\varphi$; or

	\item[(D3)] $d=2$ and two non-adjacent edges are colored by different
	colors under $\psi$, or there is one edge $e$ colored under $\varphi$
	and another edge $e'$ colored under $\psi$, such that $e$ and $e'$
	have different colors if they are adjacent, and the same color if they 
	are non-adjacent.

\end{itemize}

\begin{theorem}
\label{th:extavoid}
	Let $\varphi$ be a proper $d$-edge precoloring
	of $d-k$ edges of $Q_d$ and $\psi$ be a partial coloring of
	$k$ edges in $Q_d$, where $1 \leq k \leq d$.
	There is an extension of $\varphi$ that avoids $\psi$ unless
	some edge of $Q_d$ has the same color under $\varphi$ or $\psi$,
	or
	the colorings satisfy one of the conditions (D1)-(D3).	
\end{theorem}
\begin{proof}
	If $Q_d$ contains altogether $d-1$ edges that are colored 
	under $\varphi$ and $\psi$
	(i.e. some edge is colored under both $\varphi$ and $\psi$),
	then since at most $d-1$ edges are colored, we can
	form a new partial proper edge coloring from $\varphi$ by greedily
	assigning some color from $\{1,\dots,d\} \setminus \psi(e)$ to any edge $e$
	that 
	is colored under $\psi$, but not colored under $\varphi$, so that the
	resulting coloring $\varphi'$ is proper. By	
	Theorem \ref{extQ_d}, $\varphi'$ is extendable, so there is an extension
	of $\varphi$ that avoids $\psi$.
	
	Now assume that altogether exactly $d$ edges are colored 
	under $\varphi$ and $\psi$, so no edge is colored
	under both $\varphi$ and $\psi$. Let $E_{\varphi,\psi}$
	be the set of edges in $E(Q_d)$ that are colored under
	$\varphi$ or $\psi$.
	The case when $d \leq 2$ is trivial, so assume that $d \geq 3$.
	We shall consider some different cases.

	Suppose first that there are two non-adjacent edges $e_1$ 
	and $e_2$ that are colored under $\psi$. Then we consider the 
	coloring $\varphi'$ obtained
	from $\varphi$ by in addition coloring every $\psi$-colored
	edge in such a way that the resulting precoloring is
	proper and avoids $\psi$; note that since $e_1$ and $e_2$ are non-adjacent,
	this is possible.
	At most $d$ edges
	are colored under the resulting coloring $\varphi'$, so if it
	is not extendable, then 
	$\varphi' \in \mathcal{C}$.
	
	If $\varphi' \in \mathcal{C}_1$, then
	there is an uncolored edge $uv$ in $Q_d$ such that
		$u$ is incident with edges of $r \leq d$ distinct colors under $\varphi'$ and
		$v$ is incident to $d-r$ edges $\varphi'$-colored with $d-r$
		other distinct colors. 
		Suppose without loss of generality that $e_1$ is incident with $u$, 
		$e_2$ is incident
		with $v$ and that at least two $\varphi'$-colored edges are incident with $u$.
		Then we can define a new edge coloring of $E_{\varphi,\psi}$ from $\varphi'$
		that avoids $\psi$
		by recoloring $e_2$ by some color that appears at $u$. The obtained 
		partial edge
		coloring is not in $\mathcal{C}$, and thus there is an extension of $\varphi$
		that avoids $\psi$.
	
	If $\varphi' \in \mathcal{C}_3 \cup \mathcal{C}_4$, 
	then since all edges in $E_{\varphi,\psi}$ 
	are non-adjacent, we can recolor the edges that are
	colored under both $\varphi'$ and $\psi$ to obtain a proper coloring
	of $E_{\varphi,\psi}$ that avoids $\psi$ and is extendable to a
	proper $d$-edge coloring. Hence, there is an extension of $\varphi$
	that avoids $\psi$.
	
	Suppose now that $\varphi' \in \mathcal{C}_2$. Since $e_1$ and $e_2$
	are non-adjacent, at least one of them is not adjacent to any other
	edge from $E_{\varphi,\psi}$. Thus, we may recolor this edge and a similar
	argument as in the preceding paragraph shows that there is an extension
	of $\varphi$ that avoids $\psi$.
	
	\bigskip

	Suppose now that there are at least two edges colored under $\psi$
	and that all such edges are pairwise adjacent. 
	Thus there is some
	vertex $v$ that is incident with every edge that is colored under $\psi$.
	If we cannot define a new proper coloring $\varphi'$ of $E_{\varphi,\psi}$ 
	from $\varphi$ by coloring
	the $\psi$-colored edges in such a way that $\varphi'$ avoids $\psi$, then  
	all $\psi$-colored edges are colored by
	the same color. Moreover, if there is no such coloring $\varphi'$, then all 
	$\varphi$-colored edges are incident with $v$
	and have colors that are distinct from the $\psi$-colored edges; 
	that is, (D1) holds.
	
	Let us now consider the case when we can define a coloring $\varphi'$ as described
	in the preceding paragraph. Then $\varphi'$ is extendable, unless 
	$\varphi' \in \mathcal{C}_1 \cup \mathcal{C}_2$.
	
	Suppose first that $\varphi' \in \mathcal{C}_1$. 
	Then, since all colors in $\{1,\dots,d\}$ appear on edges 
	under $\varphi'$, there must
	be some $\varphi$-colored edge incident with $u$; suppose that such an
	edge has color $c$ under $\varphi$. Then no edge incident with $v$
	is $\varphi'$-colored $c$, because $\varphi' \in \mathcal{C}_1$.
	Now, if there is such a color $c$, such that, in addition,
	some $\psi$-colored edge $e$,
	incident with $v$,  is not colored $c$ under $\psi$, 
	then we can, from $\varphi'$, define a new coloring $\varphi''$
	of $E_{\varphi,\psi}$  by recoloring $e$ by the color $c$. 
	Since $\varphi'' \notin \mathcal{C}$, it is extendable. In conclusion,
	there
	is an extension of $\varphi$ that avoids $\psi$ unless (D1) holds.
	A similar argument applies if $\varphi' \in \mathcal{C}_2$.

	\bigskip
	
	It remains to consider the case when exactly one edge $e=uv$ is colored
	under $\psi$. 
	If we cannot pick some color for $e$ that is distinct from $\psi(e)$ and
	satisfies that this coloring of $e$ taken together with $\varphi$ is proper,
	then $\varphi$ and $\psi$ satisfy (D2).
	On the other hand, if we can define such a coloring $\varphi'$ of 
	$E_{\varphi,\psi}$ 
	from $\varphi$, which avoids $\psi$, then there is an extension of $\varphi$
	that avoids $\psi$ unless $\varphi' \in 
	\mathcal{C}$.
	
	If $\varphi' \in \mathcal{C}_3$ or $\varphi' \in \mathcal{C}_4$, then
	since all $\varphi'$-colored edges are pairwise non-adjacent, we can define
	a new proper coloring of $E_{\varphi,\psi}$ from $\varphi'$ that is extendable,
	and which avoids $\psi$.
	
	If $\varphi' \in \mathcal{C}_1$, then we can similarly define a new proper
	coloring $\varphi''$ of $E_{\varphi,\psi}$ that is extendable, unless
	exactly one $\varphi$-colored edge $e'$ is not adjacent to $e$ and
	$\varphi(e') = \psi(e)$; that is, (D1) holds.
	
	Finally, if $\varphi'$ satisfies (C2), then a similar argument shows
	that we can define a new extendable partial edge coloring of $E_{\varphi,\psi}$
	that avoids $\psi$, unless $\varphi$ and $\psi$ satisfy (D1).
\end{proof}

\end{document}